\documentclass[a4paper,10pt]{article}
\usepackage{amsmath}
\usepackage{amssymb}
\usepackage{amsthm}

\usepackage{todonotes}

\usepackage{mathrsfs}
\usepackage{graphicx,subfigure}
\usepackage{paralist} %
\usepackage{enumerate}
\usepackage{hyperref,url}
\usepackage{pgfplots}
\usetikzlibrary{backgrounds}
\usepgfplotslibrary{groupplots}
\usepackage[export]{adjustbox}

\usepackage{bm} %
\usepackage{placeins}
\usepackage{color}
\usepackage[margin=30mm]{geometry}

\usepackage{verbatim}

\usepackage{authblk}

\newcommand{\dx}{\,dx}

\newcommand{\be}{\begin{equation}}
\newcommand{\ee}{\end{equation}}
\newcommand{\been}{\begin{eqnarray*}}
\newcommand{\eeen}{\end{eqnarray*}}

\newcommand{\Omref}{\hat \Omega}

\theoremstyle{plain}
\newtheorem{thm}{Theorem}[section]

\newtheorem{lemma}[thm]{Lemma}

\theoremstyle{plain}

\newtheorem{defn}{Definition}[section]

\newtheorem{alg}{Algorithm}[section]
\numberwithin{equation}{section}

\providecommand{\keywords}[1]{\textbf{\textit{Keywords:}} #1}
\providecommand{\subjclass}[1]{\textbf{\textit{MSC subject classification:}} #1}

\usepackage{adjustbox}

\newcommand{\R}{\mathbb{R}}

\DeclareMathOperator{\Div}{div}

\DeclareMathOperator*{\argmin}{arg\,min}

\DeclareMathOperator{\id}{id}

\begin{document}

\title{Global convergence of $W^{1,\infty}$-steepest descent for PDE constrained shape optimisation with semilinear elliptic equations in function space}
\author[1]{Klaus Deckelnick}
\affil[1]{Otto-von-Guericke-University Magdeburg, Department of Mathematics, Universit\"atsplatz 2, 39106 Magdeburg}
\author[2]{Philip J.~Herbert}
\affil[2]{Department of Mathematics, University of Sussex, Brighton, BN1 9RF, United Kingdom}
\author[3]{Michael Hinze%
}
\affil[3]{Mathematical Institute, University of Koblenz, Universit\"atsstr. 1, D-56070 Koblenz}
\date{\today}

\maketitle
\begin{abstract}
We prove global convergence in function space for the steepest descent
method in shape optimisation with semilinear elliptic partial
differential equations. Steepest descent is realized in the Lipschitz
topology. In addition, we prove a conditional convergence result for the
resulting shapes in two space dimensions.
\end{abstract}

\noindent\keywords{PDE constrained shape optimisation, 
$W^{1,\infty}$-steepest-descent, global convergence, semilinear elliptic PDE} \\[2mm]
\noindent\subjclass{ 35Q93, 49Q10, 49J20}
\\
\renewcommand{\thefootnote}{\arabic{footnote}}

\section{Introduction}\label{intro}
In this article, we continue our investigations on the convergence of the steepest descent method for shape optimisation problems, which we began in \cite{DHH25-convergence}. To this end, we consider the general PDE constrained shape optimisation problem 

\begin{equation}\label{prob0}
    \min \mathcal{J}(\Omega):= \int_\Omega j(x, u(x),\nabla u(x))  \dx,\, \Omega \in \mathcal{S},
\end{equation}
where $j$ is a real-valued function whose properties will be specified in Section \ref{sec:Preliminaries}, and $u$ now weakly solves the semilinear elliptic boundary value problem
\[
-\Delta u + g(u) =f \text{ in } \Omega, \quad u=0 \text{ on } \partial \Omega.
\]
Furthermore, $\mathcal{S}$ is a collection of admissible domains contained in a given hold-all domain $D \subset \mathbb R^d$ for $d=2,3$. As it is in general hard to calculate a minimizer of \eqref{prob0} one frequently uses
a descent approach in order to produce a sequence $(\Omega^k)_{k \in \mathbb N_0} \subset \mathcal S$
which, under appropriate conditions, converges to a stationary shape $\Omega$ satisfying $\mathcal J'(\Omega)=0$. We shall use the method of mappings and construct the sets $\Omega^k \in \mathcal S$ 
as $\Omega^k=\Phi^k(\Omref)$, where $\Phi^k\colon \bar D \rightarrow \bar D$ is a bi--Lipschitz mapping  and $\Omref \Subset D$ is a fixed reference domain. Under suitable conditions on $j$ the shape derivative
$\mathcal J'(\Omega^k)$ is a bounded linear functional on $W^{1,\infty}_0(D,\mathbb R^d)$ so that
a natural choice of descent direction is given by 
\begin{equation}\label{eq:generalMinimisationForDirection}
    V^k \in \argmin \left\{  \mathcal{J}'(\Omega^k)[V] :  V\in  W^{1,\infty}_0(D,\mathbb R^d), | D V | \leq 1 \mbox{ a.e. in } D \right\}.
\end{equation}
Here, $|DV|$ denotes the spectral norm of the Jacobian $DV$. As a result, $\mathcal{J}'(\Omega^k)[V^k]<0$ unless $\Omega^k$ is already stationary. 
The new domain $\Omega^{k+1}$ is then obtained as $\Omega^{k+1} = \Phi^{k+1}(\Omref)$, where $\Phi^{k+1} = ({\rm id} + t_k V^k)\circ \Phi^k$ and $t_k>0$ is chosen via the Armijo step size rule, see Algorithm \ref{alg:Steepest}
for a more detailed description. 
Our aim in this paper is to investigate the convergence of this algorithm within a function
space setting. As our main results we shall prove that
\begin{itemize}
\item an Armijo stepsize for $W^{1,\infty}$-steepest descent directions of the shape functional exists, see Lemma \ref{decrease};
\item the algorithm is globally convergent in the sense that $\Vert \mathcal J'(\Omega^k) \Vert \rightarrow 0$ as $k \rightarrow \infty$, see Theorem \ref{conv1};
\item in two space dimensions the sequence $(\Omega^k)_{k \in \mathbb N_0}$ has a subsequence
that converges with respect to the Hausdorff complementary metric to a stationary point of $\mathcal J$,
provided that $(\Phi^k)_{k \in \mathbb N_0}$ is bounded in $W^{1,\infty}(D,\mathbb R^2)$, see Theorem \ref{conv2}. 
\end{itemize}
These results build on and extend the analysis in \cite{DHH25-convergence}, where
shape optimisation subject to the linear Poisson problem is considered and the convergence for a finite element approximation of problem \eqref{prob0} is examined.  Theorem 3.3 in \cite{DHH25-convergence} proves global convergence of the steepest descent method applied to the discretized problem, and, under an additional condition, the existence of a  discrete stationary shape.  A crucial ingredient in the corresponding proof   is \cite[Lemma 3.2]{DHH25-convergence}, which guarantees the existence of the Armijo stepsize in the steepest descent algorithm. A close inspection of that result shows that it is independent of the finite element discretization parameter and therefore virtually carries over to our infinite-dimensional setting. We shall exploit this observation and  extend it to the case of a semilinear state equation in order
to prove that our steepest descent method is globally convergent, see 
Theorem \ref{conv1} below. \\

Of course  it is also interesting to examine under which conditions
the sequence $(\Omega^k)_{k \in \mathbb N_0}$ converges to a stationary shape $\Omega$. A related result
establishing the convergence of a sequence of discrete stationary shapes to a stationary shape of
$\mathcal J$ was obtained  in \cite[Theorem 4.7]{DHH25-convergence} under the assumption that the mappings $(\Phi^k)_{k \in \mathbb N_0}$ satisfy
\begin{equation} \label{eq:assump}
 \frac{1}{M} |x_1 - x_2 | \leq | \Phi^k(x_1) - \Phi^k(x_2) | \leq M | x_1 - x_2 | \quad \mbox{ for all } x_1,x_2 \in D \mbox{ and all } k \in \mathbb N_0
\end{equation}
for some $M>0$. We shall be able to weaken this condition in the case of two space dimensions in that we only
assume the second inequality in \eqref{eq:assump}, see Theorem \ref{conv2}. 
The reason for the  restriction  to two dimensions is that we make strong use
of two important results due to {\v{S}}ver{\'a}k, \cite{vsverak1993optimal} and Chambolle \& Doveri, \cite{ChamDov97} respectively.  These results establish $\gamma$--convergence for the Dirichlet and Neumann problems and are crucial for passing to the limit
in the state and adjoint equations as well as  in the expression for $\mathcal J'(\Omega)$. An important
advantage of these results is that they essentially require topological conditions on the 
sequence $(\Omega^k)_{k \in \mathbb N_0}$  which in our case follow from the representation $\Omega^k=\Phi^k(\Omref)$ together
with the fact that the mappings  $\Phi^k$ are bi-Lipschitz. \\
 For additional information on the subject of shape optimisation we refer the reader to the seminal works of Delfour and Zol\'esio \cite{DelZol11},  of Sokolowski and Zol\'esio \cite{SZ92}, and the recent overview article \cite{ADJ21} by Allaire, Dapogny, and Jouve, where also a comprehensive  bibliography on the topic can be found.

\section{Preliminaries}\label{sec:Preliminaries}
\subsection{Setting of the problem}

\noindent
Let $D \subset \mathbb R^d$ be an open, convex, and bounded hold-all domain with a Lipschitz boundary and 
\begin{displaymath}
\mathcal S:= \lbrace \Omega \subset D \, | \, \Omega \mbox{ is open } \rbrace.
\end{displaymath}
\noindent
We consider the shape optimisation problem
\begin{displaymath}
\min_{\Omega \in \mathcal S} \mathcal J(\Omega) = \int_{\Omega} j(\cdot,u,\nabla u)  \dx,
\end{displaymath}
where $u \in H^1_0(\Omega)\cap L^\infty(\Omega)$ is the unique solution of 
\begin{equation} \label{state}
\int_{\Omega} \nabla u \cdot \nabla \eta  \dx + \int_{\Omega} g(u) \eta \dx  = \int_{\Omega }f \eta \dx  \qquad \mbox{ for all } \eta \in H^1_0(\Omega).
\end{equation}
Here  '$\cdot$' denotes the Euclidean inner product of two vectors. 
In what follows we assume that $f \in H^1(D)$, $g \in C^2(\mathbb{R})$ with $g'(t) \ge 0$ for all $t \in \mathbb R$ and that $j \in C^2(D \times \mathbb R \times \mathbb R^d)$ satisfies 
\begin{eqnarray}
| j(x,u,z) | + | j_x(x,u,z)| + | j_{xx}(x,u,z)|  & \leq  & \varphi_1(x)+ c_1 \bigl( | u|^q + |z|^2 \bigr); \label{jest} \\
| j_u(x,u,z) | + | j_{xu}(x,u,z) | & \leq  &  \varphi_2(x) + c_2 \bigl(  |u |^{q-1} + | z|^{2-\frac{2}{q}} \bigr);   \label{j1est} \\
| j_z(x,u,z) | + | j_{xz}(x,u,z) |  & \leq   & \varphi_3(x) + c_3 \bigl( |u|^{\frac{q}{2}} + |z| \bigr);  \label{j2est}   \\
| j_{uu}(x,u,z) | & \leq  &\varphi_4(x) + c_4 \bigl( |u|^{q-2} + |z|^{2- \frac{4}{q}} \bigr); \label{j3est} \\
  | j_{zz}(x,u,z) | & \leq & \varphi_5(x) \label{j4est}
\end{eqnarray}
for all $(x,u,z) \in  D \times  \mathbb R \times \mathbb R^d$, {where a subscript $x$ denotes the derivative with respect to the first variable, $u$, the second, and $z$, the third}.
Here, $2 \leq q < \infty$ if $d=2$, and $q= 6$ if $d =3$. Also, $\varphi_1,\ldots,\varphi_5$ are non-negative with 
$\varphi_1 \in L^1(D), \varphi_2 \in L^{\frac{q}{q-1}}(D), \varphi_3 \in L^2(D), \varphi_4 \in L^{\frac{q}{q-2}}(D)$ and 
$\varphi_5 \in L^\infty(D)$.
{Note that the choice of $q$ implies the continuous embedding $H^1_0(D) \hookrightarrow L^q(D)$, so that there exists $c>0$ with
\begin{equation} \label{Dembed}
\Vert v \Vert_{L^q} \leq c \Vert v \Vert_{H^1} \qquad \mbox{ for all } v \in H^1_0(D).
\end{equation} } 
\noindent
From here onwards we omit the explicit dependencies on the spatial variable $x$ wherever appropriate. 
Using the techniques described in Sections 4.4 and 4.5 of \cite{ADJ21}, one calculates for the shape  derivative of $\mathcal J$
\begin{eqnarray}
 \mathcal J'(\Omega)[V] 
 & = &  \int_{\Omega} \Bigl( j(\cdot,u,\nabla u) \Div V + j_x(\cdot,u,\nabla u)\cdot  V  - j_z(\cdot,u,\nabla u) \cdot DV^{\mathsf{T}} \nabla u \Bigr)  \dx    \label{firstvar} \\
 & & + \int_{\Omega} \Bigl(  \bigl( DV + D V^{\mathsf{T}} - \Div V I \bigr)  \nabla u \cdot \nabla p - g(u)p \mbox{ div }V + {\Div}( f V) p   \Bigr)  \dx \nonumber
\end{eqnarray}
for all $V \in W^{1,\infty}_0(D, \mathbb R^d)$. Here,
$p \in H^1_0(\Omega)$ is the solution of the adjoint problem
\begin{equation} \label{adj}
\int_{\Omega} \nabla p \cdot \nabla \eta  \dx + \int_{\Omega} g'(u) p \eta \dx = \int_{\Omega} \bigl( j_u(\cdot,u,\nabla u) \eta + j_z(\cdot,u,\nabla u) \cdot \nabla \eta \bigr)  \dx \quad \mbox{ for all } \eta \in H^1_0(\Omega).
\end{equation}
Note that the assumptions \eqref{jest}--\eqref{j2est} ensure that
all integrals that appear on the right hand side of \eqref{firstvar} and \eqref{adj} exist. At the same time these growth conditions
also guarantee uniform control on the solutions of the state
and adjoint equations as shown in 
\begin{lemma} \label{lem:apriori} Let $\Omega \in \mathcal S$. Then the problems \eqref{state} and \eqref{adj} admit unique solutions $u\in H^1_0(\Omega)\cap L^\infty(\Omega)$, and $p \in H^1_0(\Omega)$, respectively. Moreover, there exists a constant $c^\star$, such that
\begin{equation} \label{discapriori}
\Vert u \Vert_{H^1} + \Vert u\Vert_{L^\infty} + \Vert p \Vert_{H^1} \leq c^\star.
\end{equation}
The constant $c^*$ only depends on $d,D, f,j$ and $g$, but
is independent of $\Omega$. Furthermore, we think of $u$ and $p$ as being extended by zero to $D$.
\end{lemma}
\begin{proof} In order to prove the existence of a solution $u \in H^1_0(\Omega)$ of
\eqref{state} we first consider for $M>0$ the modified nonlinearity
\begin{displaymath}
g_M(t):= \left\{
\begin{array}{cl}
g(-M), & t<-M, \\
g(t), & -M \leq t \leq M, \\
g(M), & t>M.
\end{array}
\right.
\end{displaymath}
Note that $g_M$ is increasing since $g' \geq 0$. 
Using the theory of monotone operators it can be shown that the problem
\begin{displaymath}
\int_\Omega \nabla u_M \cdot \nabla \eta \dx + \int_\Omega g_M(u_M) \eta \dx =
\int_\Omega f \eta \dx \qquad \forall \eta \in H^1_0(\Omega)
\end{displaymath}
has a unique solution $u_M \in H^1_0(\Omega)$, see \cite[Section 2.6]{VM25} for details.  Let us next denote by
$v \in H^2(D) \cap H^1_0(D)$ the solution to the Poisson problem
\begin{displaymath}
- \Delta v = f - g(0) \; \mbox{ a.e. in }  D, \quad v=0 \mbox{ on }  \partial D.  
\end{displaymath}
Since $H^2(D) \hookrightarrow C^0(\bar D)$ for $d=2,3$ we infer that there exists $K \geq 0$, which only depends on $d,D,f$ and $g$,
such that $| v(x) | \leq K$ for all $x \in \bar D$. The function $w:=v+K$ then satisfies
$w \geq 0$ in $\bar D$ as well as
\begin{displaymath}
- \Delta w + g_M(w) = -\Delta v + g_M(w) = f - g_M(0) + g_M(w) \geq f \quad \mbox{ a.e. in } D
\end{displaymath}
since $g_M$ is increasing and $g_M(0)=g(0)$. If we multiply the above relation by
$\eta \in C^1_0(\Omega), \eta \geq 0$ in  $\Omega$ and use an approximation argument we obtain
\begin{displaymath}
    \int_\Omega \nabla w \cdot \nabla \eta \dx + \int_\Omega g_M(w) \eta \dx \geq
    \int_\Omega f \eta \dx \qquad \forall \eta \in H^1_0(\Omega), \, \eta \geq 0 \mbox{ a.e. in } \Omega.
\end{displaymath}
Hence,
\begin{displaymath}
    \int_\Omega \nabla( u_M-w) \cdot \nabla \eta \dx + \int_\Omega (g_M(u_M)-g_M(w)) \eta \dx \leq
0 \qquad \forall \eta \in H^1_0(\Omega), \eta \geq 0 \mbox{ a.e. in } \Omega.
\end{displaymath}
Since $w \geq 0$ in $D$ we have that  $\eta =(u_M-w)^+ \in H^1_0(\Omega), \eta \geq 0$ and
therefore by the monotonicity of $g_M$
\begin{displaymath}
    \int_\Omega | \nabla(u_M-w)^+ |^2 \dx \leq 0.
\end{displaymath}
Thus, $(u_M - w)^+=0$ so that $u_M \leq w =v+K \leq 2K$ a.e. in $\Omega$. A similar argument
shows that $u_M \geq -2K$ a.e. in $\Omega$, so that 
\begin{equation} \label{eq:umbound}
\Vert  u_M \Vert_{L^\infty}  \leq M
\end{equation}
if we choose $M=2K$. In particular, $g_M(u_M)=g(u_M)$ and $u=u_M$ is a solution of
\eqref{state}, while uniqueness is a consequence of the monotonicity of $g$. While the $L^\infty$--bound for $u$ in \eqref{discapriori} then follows from \eqref{eq:umbound}, the $H^1$--seminorm is controlled
by testing \eqref{state} with $\eta=u$ and using that $g(u)u
\geq g(0)u$ a.e. in $\Omega$. \\
In order to prove the bound on $p$ we test \eqref{adj} with $\eta=p \in H^1_0(\Omega)$ and use 
the fact that $g' \geq 0$, \eqref{j1est}, \eqref{j2est},
H\"older's inequality and  {\eqref{Dembed}} to obtain
\begin{eqnarray*}
\lefteqn{  \int_D | \nabla p |^2  \dx = 
\int_{\Omega} | \nabla p |^2  \dx    } \\
 & \leq & \int_{\Omega} \left[ \bigl( \varphi_2 + c_2( | u |^{q-1} + | \nabla u |^{\frac{2(q-1)}{q}} ) \bigr) | p | + \bigl( \varphi_3 + c_3( | u|^{\frac{q}{2}} + | \nabla u | ) \bigr)  | \nabla p | 
 \right] \dx \\
& \leq & \bigl( \Vert \varphi_2 \Vert_{L^{\frac{q}{q-1}}} + c_2(\Vert u \Vert_{L^q}^{q-1} + \Vert \nabla u \Vert_{L^2}^{\frac{2(q-1)}{q}} ) \bigr)  \Vert p \Vert_{L^q}   \\
& & + \bigl( \Vert \varphi_3 \Vert_{L^2} + c_3( \Vert u \Vert_{L^q}^{\frac{q}{2}} + \Vert \nabla u \Vert_{L^2})  \bigr) \Vert  \nabla p \Vert_{L^2} \\[2mm]
& \leq & c \bigl( 1+ \Vert u \Vert_{H^1}^{q-1}  \bigr)   \Vert p \Vert_{H^1} \leq c \bigl( 1 + (c^*)^{q-1} \bigr) \Vert p \Vert_{H^1} \leq c \bigl( 1 + (c^*)^{q-1} \bigr) \Vert \nabla p \Vert_{L^2},
\end{eqnarray*}
{where we also made use of Poincar\'e's inequality for $D$ and the bound on $u$.} The estimate for 
$\Vert p \Vert_{L^2}$ now follows from another application  of Poincar\'e's inequality.
\end{proof}
\noindent

\subsection{Descent algorithm} 

\noindent
Let us next translate the ideas formulated in the introduction
into a descent method in the space $W^{1,\infty}(D,\mathbb R^d)$.
To do so, let us fix an open, nonempty reference domain $\Omref \Subset D$. The open sets  generated by
our descent method will be constructed via the method of mappings
in the form $\Omega^k = \Phi^k(\Omref)$ with suitable 
bi--Lipschitz mappings  $\Phi^k:\bar D \rightarrow \bar D$.

\begin{alg}[Steepest descent with Armijo line search]\label{alg:Steepest}
Let $\gamma \in (0,1)$ be a fixed constant. \\ [2mm]
0. Let $\Omega^0:= \Omref, \Phi^0:= \id$. \\[2mm]
For k=0,1,2,\ldots: \\[2mm]
1. If $\mathcal J'(\Omega^k)=0$, then stop. \\[2mm]
2. Choose $V^k \in  W^{1,\infty}_0(D, \mathbb R^d)$ such that 
\begin{displaymath}
V^k \in \argmin \lbrace J'(\Omega^k)[W] \, | \, W \in W^{1,\infty}_0(D, \mathbb R^d), \, | DW | \leq 1 \mbox{ a.e. in } D \rbrace.
\end{displaymath}
3. Choose the maximum $t_k \in \lbrace  \frac{1}{2}, \frac{1}{4}, \ldots \rbrace$ such that 
\begin{displaymath}
\mathcal J \bigl( (\id + t_k V^k)(\Omega^k) \bigr) - \mathcal J(\Omega^k) \leq \gamma t_k \mathcal J'(\Omega^k)[V^k].
\end{displaymath}
4. Set $\Phi^{k+1}:= (\id + t_k V^k )\circ \Phi^k, \, \Omega^{k+1}:= (\id+ t_k V^k)(\Omega^k)$.
\end{alg}

\vspace{2mm}
\noindent
In the following lemma we collect the properties of the mappings $\Phi^k$ generated by Algorithm \ref{alg:Steepest}.

\begin{lemma} \label{lem:propalg}
Algorithm \ref{alg:Steepest} produces a sequence of bi--Lipschitz mappings $\Phi^k: \bar D \rightarrow \bar D$ with
$\Phi^k=\id$ on $\partial D$ and open sets $\Omega^k \subset D$ with $\Omega^k= \Phi^k(\Omref)$.
\end{lemma}
\begin{proof} Let us prove these assertions by induction over $k \in \mathbb N_0$. They are clearly
satisfied for $k=0$ by step 0 of the algorithm.  Let us
next assume that for some $k \in \mathbb N_0$ the mapping $\Phi^k:
\bar D \rightarrow \bar D$ is bi--Lipschitz with $\Phi^k=\id$ on $\partial D$ and $\Omega^k=\Phi^k(\Omref)$. 
Let $V^k \in W^{1,\infty}_0(D,\mathbb R^d)$ be the direction obtained in step 2 of the algorithm (see
Theorem 2.1 and the subsequent remarks in \cite{Deckelnick26122024} for the existence of $V^k$) and $0 \leq t_k \leq \frac{1}{2}$ the step size chosen in step 3.   Since $
 | DV^k | \leq 1$ a.e. in $D$ and $D$ is convex we infer that 
\begin{displaymath}
\frac{1}{2} | x_1 -x_2 | \leq | (\id+t_k V^k)(x_1) - (\id+t_k V^k)(x_2)| \leq \frac{3}{2} | x_1 - x_2 | \qquad \forall x_1,x_2 \in D,
\end{displaymath}
so that $\Phi^{k+1}=(\id+t_k V^k) \circ \Phi^k$ is bi--Lipschitz and in particular injective.
Furthermore, observing that  $V^k=0$ on $\partial D$ we have  that $\Phi^{k+1}=\Phi^k=\mbox{id}$ on $\partial D$. \\
In order to prove that $\Phi^{k+1}$ is surjective we repeat an argument that was
used  in the proof of Lemma 2.1 in \cite{DHH25-convergence}. For
every $p \in D$ we have that
\begin{displaymath}
\mbox{deg}(\Phi^{k+1},D,p) = \mbox{deg}(\mbox{id},D,p)=1,
\end{displaymath}
where $\mbox{deg}$ denotes the Brouwer degree. Therefore we infer from the existence property of the degree that there is $x \in D$ with $p=\Phi^{k+1}(x)$ so
that $D \subset \Phi^{k+1}(D)$. Next we claim that $D$ is closed in $\Phi^{k+1}(D)$. To see this,
let $(p_n)_{n \in \mathbb N}$ be a sequence in $D$ such that $p_n \rightarrow p$ as $n \rightarrow \infty$ for some $p \in \Phi^{k+1}(D)$, say $p=\Phi^{k+1}(x)$ with
$x \in D$. If $p \in \partial D$, then $\Phi^{k+1}(x)=p=\Phi^{k+1}(p)$, which implies in view of the  injectivity of $\Phi^{k+1}$ that $x=p$,
a contradiction. Hence $p \in D$. As $D$ is also open in $\Phi^{k+1}(D)$ and $\Phi^{k+1}(D)$ is connected we infer that $D=\Phi^{k+1}(D)$. Since  $\Phi^{k+1}=\id$ on $\partial D$ we conclude that $\Phi^{k+1}:\bar D \rightarrow \bar D$ is bijective. Therefore, 
$\Phi^{k+1}$ is a bi--Lipschitz map  from $\bar D$ to $\bar D$ with $\Phi^{k+1}=\id$ on $\partial D$. Furthermore,
$\Omega^{k+1}=(\id+t_k V^k)(\Omega^k)=(\id+t_k V^k) \circ \Phi^k(\Omref)=\Phi^{k+1}(\Omref)$. \\
Finally, the sets  $\Omega^{k}=\Phi^k(\Omref)$  are open because $\Phi^k$ is a homeomorphism and
$\Omref$ is open. 
\end{proof}

\section{Convergence of the descent algorithm}\label{sec:convergenceOfAlgorithm}
In this Section, we now demonstrate the global convergence of the algorithm, as well as showing that in two dimensions, the sequence $(\Omega^k)_{k \in \mathbb{N}_0}$ has a convergent subsequence provided that the sequence $(\Phi^k)_{k \in \mathbb{N}_0}$ is bounded in $W^{1,\infty}(D;\R^2)$.

\subsection{Global convergence}
\noindent
Our aim in this section is to  show that the steepest descent Algorithm \ref{alg:Steepest} is globally convergent in the sense that
\begin{displaymath}
\Vert \mathcal J'(\Omega^k) \Vert: = \sup \lbrace \mathcal J'(\Omega^k)[W] \, | \, W \in W^{1,\infty}_0(D, \mathbb R^d), | DW | \leq 1 \mbox{ a.e. in } D \rbrace \rightarrow 0, \quad \mbox{ as }
k \rightarrow \infty.
\end{displaymath}
The corresponding result mainly relies on the following lemma,
which establishes the existence of an Armijo stepsize for a 
given descent direction. Let us note that it is here, where
the growth conditions  on the second partial derivatives of $j$ in \eqref{jest}--\eqref{j4est} are needed. 

\begin{lemma} \label{decrease}
 Let {$\gamma \in (0,1)$ be a fixed constant,} $\Omega \in \mathcal S$, and
\begin{displaymath}
V \in \argmin \lbrace J'(\Omega)[W] \, | \, W \in W^{1,\infty}_0(D, \mathbb R^d), \, | DW | \leq 1 \mbox{ a.e. in } D \rbrace.
\end{displaymath}
Suppose that $\mathcal J'(\Omega)[V] \leq - \epsilon$ for some $\epsilon>0$. Then there exists 
$0< \delta <1$ which only depends on $j, f, g, D, d, \gamma$ and $\epsilon$ such that
\begin{displaymath}
\mathcal J(\Omega_{t}) - \mathcal J(\Omega) \leq \gamma t \mathcal J'(\Omega)[V] \qquad \mbox{ for all } 0 \leq t \leq \delta, \; \mbox{ where } \Omega_{t} = T_t(\Omega) \mbox{ and } T_t=\mbox{id}+ t V.
\end{displaymath}
\end{lemma}
\begin{proof}
The proof follows the lines of the proof of \cite[Lemma 3.2]{DHH25-convergence}, in which
a corresponding result was obtained in a finite--dimensional setting, where
state, adjoint state and deformation field are elements of finite element spaces. 
Nevertheless many arguments immediately carry over to the infinite--dimensional case and
we will direct the reader to the corresponding argument in \cite{DHH25-convergence}
when appropriate. \\
Let $\Omega_t=T_t(\Omega)$ with $T_t(x)=x+t V(x)$. In view of the definition of $\mathcal J$
we have
\begin{displaymath}
\mathcal J(\Omega_t) = \int_{\Omega_t} j(\cdot,u_t,\nabla u_t)  \dx,
\end{displaymath}
where $u_t \in H^1_0(\Omega_t)$ solves
\begin{displaymath}  
\int_{\Omega_t} \nabla u_t \cdot \nabla \eta_t  \dx + \int_{\Omega_t} g(u_t) \eta_t \dx = \int_{\Omega_t} f \eta_t  \dx \qquad \forall \eta_t \in H^1_0(\Omega_t).
\end{displaymath}
Choosing $\eta_t:= \eta \circ T_t^{-1} \in H^1_0(\Omega_t)$ for a fixed $\eta \in H^1_0(\Omega)$ we thus infer that 
\begin{equation}  \label{hatuht}
\int_{\Omega_t} \nabla u_t \cdot \nabla (\eta \circ T_t^{-1})  \dx + \int_{\Omega_t} g(u_t) \eta \circ T_t^{-1} \dx = \int_{\Omega_t} f \eta \circ T_t^{-1}  \dx \qquad
\forall  \eta \in H^1_0(\Omega).
\end{equation}
Transforming to $\Omega$ we obtain
\begin{equation} \label{uhtrel}
\int_{\Omega} \nabla u_t \circ T_t \cdot \nabla (\eta \circ T_t^{-1}) \circ T_t \, | \mbox{det} D T_t |  \dx + \int_{\Omega} g(u_t\circ T_t) \eta \, | \mbox{det} D T_t |dx =
\int_{\Omega} f \circ T_t \, \eta \, | \mbox{det} D T_t |  \dx 
\end{equation}
for all $\eta \in H^1_0(\Omega)$. Since $\mbox{det} DT_t = \mbox{det}(I+ t DV)$ and $| D V | \leq 1$ in $\bar D$ one easily
verifies that 
\begin{equation} \label{dif1}
\mbox{det} DT_t -1 = t {\Div} V + r_1, \quad \mbox{ with } | r_1 | \leq c t^2,
\end{equation}
where the constant $c$ only depends on $d$. This also implies that  there exists $\delta_1>0$ such
that $\mbox{det} DT_t >0, 0 \leq t \leq \delta_1$. Defining $A_t:= (DT_t)^{-1} (D T_t)^{- \mathsf{T}} \mbox{det} DT_t$ we hence see that the function
$\hat u_t:= u_t \circ T_t \in H^1_0(\Omega)$ satisfies
\begin{equation} \label{stateuht}
\int_{\Omega} A_t \nabla \hat u_t \cdot \nabla \eta  \dx + \int_{\Omega} g(\hat u_t) \eta \, \mbox{det} D T_t \dx = \int_{\Omega} f \circ T_t \, \eta \, \mbox{det} D T_t  \dx \qquad
\mbox{ for all } \eta \in H^1_0(\Omega).
\end{equation}
As a result we deduce that
\begin{eqnarray}
\mathcal J(\Omega_t) - \mathcal J(\Omega) 
& = & \int_{\Omega} \bigl( j(T_t,\hat u_t, D T_t^{- \mathsf{T}} \nabla \hat u_t) \, \mbox{det} DT_t - j(\cdot,u,\nabla u) \bigr)  \dx   \nonumber \\
& = & \int_{\Omega} j(\cdot,u,\nabla u) (\mbox{det} DT_t -1 )  \dx + \int_{\Omega} \bigl( j(T_t,\hat u_t,DT_t^{- \mathsf{T}} \nabla \hat u_t) - j(\cdot,u,\nabla u) \bigr)  \dx    \nonumber  \\
& & +  \int_{\Omega} \bigl( j(T_t,\hat u_t, DT_t^{- \mathsf{T}} \nabla \hat u_t) -j(\cdot,u,\nabla u) \bigr) ( \mbox{det} DT_t - 1)  \dx  \nonumber \\
& = &  \sum_{j=1}^3 T_j.  \label{difj}
\end{eqnarray}
Arguing as in \cite[(3.7)]{DHH25-convergence} and using Lemma \ref{lem:apriori} we obtain
\begin{equation} \label{eq:t1}
T_1   \leq t \int_{\Omega} j(\cdot,u,\nabla u) \Div V  \dx + c t^2.
\end{equation} 
The term $T_2$ is written in the same way as in \cite[(3.8)]{DHH25-convergence} so that
\begin{eqnarray}
T_2 & = & t  \int_{\Omega} j_x(\cdot,u,\nabla u) \cdot V  \dx - t \int_{\Omega} j_z(\cdot,u,\nabla u) \cdot DV^{\mathsf{T}} \nabla u \nonumber \\
& & + \int_{\Omega} j_z(\cdot,u,\nabla u) \cdot \bigl( (DT_t^{-\mathsf{T}} -I + t DV^{\mathsf{T}} ) \nabla \hat u_t + t DV^{\mathsf{T}} \nabla (u - \hat u_t) \bigr)  \dx \nonumber \\
& & +  \int_{\Omega} \bigl( j_u(\cdot,u,\nabla u)(\hat u_t -u)  + j_z(\cdot,u,\nabla u) \cdot \nabla (\hat u_t - u) \bigr)  \dx \nonumber \\
& & + \int_{\Omega}  \int_0^1 (1-s) \frac{d^2}{ds^2} \left[ j(\cdot +stV,s \hat u_t +(1-s) u,
s  DT_t^{-\mathsf{T}} \nabla \hat u_t +(1-s) \nabla u) \right] ds  \dx \nonumber \\
& = & \sum_{j=1}^5 T_{2,j}. \label{t2}
\end{eqnarray}
The integrals $T_{2,3}$ and $T_{2,5}$ are handled in the same way as in \cite[(3.9), (3.17)]{DHH25-convergence} and therefore
\begin{equation} \label{eq:t23}
T_{2,3} +T_{2,5}  \leq c \bigl(  t^2 + c \Vert \hat u_t - u \Vert_{H^1}^2 \bigr). 
\end{equation} 
In order to treat $T_{2,4}$ we use  \eqref{adj}, \eqref{state} and \eqref{stateuht} and write
\begin{eqnarray*}
T_{2,4} &=& \int_{\Omega} \nabla p \cdot \nabla ( \hat u_t - u )  \dx + \int_{\Omega} g'(u) p (\hat u_t-u) \dx = \int_{\Omega} \nabla p \cdot \nabla \hat u_t  \dx +  \int_{\Omega}g(\hat u_t)p \dx \\
& & - \int_{\Omega} \nabla p \cdot \nabla u  \dx  - \int_{\Omega}g(u)p \dx 
 + \int_{\Omega} \bigl( g'(u)( \hat u_t -u) - g(\hat u_t) + g(u)) p \dx  \\
 & = & \int_{\Omega} A_t \nabla \hat u_t \cdot \nabla p \dx + \int_{\Omega} g(\hat u_t) p \, 
 \mbox{det} D T_t \dx - \int_{\Omega} \nabla p \cdot \nabla u  \dx  - \int_{\Omega}g(u)p \dx \\
 & & + \int_{\Omega} (I -A_t) \nabla \hat u_t \cdot \nabla p \dx + \int_{\Omega} g(\hat u_t) 
 (1 - \mbox{det} D T_t) p \dx \\
 & & + \int_{\Omega} \bigl( g'(u)( \hat u_t -u) - g(\hat u_t) + g(u)) p \dx  \\
 & = & \int_{\Omega}  \bigl( f \circ T_t \, \mbox{det} D T_t -f \bigr) p \dx
+ \int_{\Omega} (I -  A_t) \nabla \hat u_t \cdot \nabla p \dx + \int_{\Omega} g(\hat u_t) 
 (1 - \mbox{det} D T_t) p \dx \\
 & & + \int_{\Omega} \bigl( g'(u)( \hat u_t -u) - g(\hat u_t) + g(u)) p \dx=: \sum_{k=1}^4 \tilde T_k.
 \end{eqnarray*}
 Arguing as in \cite[(3.14)]{DHH25-convergence} one shows that 
\begin{equation} \label{eq:tildet1} 
\tilde T_1  \leq t \int_{\Omega} \Div( f V) p  \dx +  c t^2 + c t  \sup_{0 \leq \sigma \leq t} \Vert \nabla f \circ T_\sigma - \nabla f \Vert_{L^2}.  
\end{equation}
 Since  $A_t = (DT_t)^{-1} (DT_t)^{-\mathsf{T}} \mbox{det} DT_t$ and $| DV| \leq 1$
 one verifies that
\begin{equation} \label{idmat}
I - A_t = t \bigl( D V + D V^{\mathsf{T}} - \Div V I  \bigr) + R_2, \qquad \mbox{ with } | R_2  | \leq c t^2,
\end{equation}
where $c$ only depends on $d$. Therefore
\begin{eqnarray} 
\tilde T_2 & = & \int_{\Omega} (I- A_t) \nabla p \cdot \nabla u  \dx + \int_{\Omega} (I - A_t) \nabla p \cdot \nabla ( \hat u_t - u)  \dx \nonumber \\
& = & 
t \int_{\Omega} \bigl( D V + D V^{\mathsf{T}} - \Div V I  \bigr) \nabla u \cdot \nabla p  \dx 
+ \int_{\Omega} R_2 \nabla u \cdot \nabla p  \dx \nonumber \\
& & + \int_{\Omega} (I - A_t) \nabla p \cdot \nabla ( \hat u_t - u)  \dx 
\nonumber\\
& \leq &  t \int_{\Omega} \bigl( D V + D V^{\mathsf{T}} - \Div V I  \bigr) \nabla u \cdot \nabla p  \dx + c \Vert \nabla p \Vert_{L^2} \bigl( t^2 +  t \Vert \hat u_t - u \Vert_{H^1} \bigr),  \label{eq:tildet2}
\end{eqnarray}
where we also used  \eqref{discapriori}. In order to deal with $\tilde T_3$ we use \eqref{dif1} and \eqref{discapriori} to obtain
\begin{eqnarray}
\tilde T_3 & =& \int_{\Omega} g(u) p (1 - \mbox{det} DT_t) \dx + \int_{\Omega} (g(\hat u_t - g(u)) p (1 - \mbox{det} D T_t) \dx \nonumber \\
& \leq & -t\int_{\Omega} g(u)p \mbox{div}V \dx + c t^2+ c t  \Vert \hat u_t - u \Vert_{L^2} \Vert p \Vert_{L^2} \nonumber \\
& \leq &  -t\int_{\Omega} g(u)p \mbox{div}V \dx + c t^2+ c \Vert \hat u_t - u \Vert_{L^2}^2. \label{eq:tilde3}
\end{eqnarray}
Finally, Taylor expansion yields
\[
g(\hat u_t)-g(u) = g'(u)(\hat u_t-u) + \int\limits_0^1 (1-s)g''(u+s(\hat u_t-u))ds (\hat u_t-u)^2
\]
and hence
\begin{eqnarray}
\tilde T_4 & = & 
- \int_{\Omega} p \int\limits_0^1 (1-s)g''(u+s(\hat u_t-u))ds (\hat u_t-u)^2 \dx \label{eq:tilde4} \\
& \leq & c 
 \int_{\Omega} \vert \hat u_t-u\vert^2  \vert p\vert \dx \le c 
 \Vert \hat u_t -u\Vert_{L^4}^2  \Vert p\Vert_{L^2} \leq  c \Vert \hat u_t-u\Vert_{H^1}^2. \nonumber
\end{eqnarray}
Collecting the above terms we have
\begin{eqnarray}
T_{2,4} & \leq &  t \int_{\Omega} \bigl( D V + D V^{\mathsf{T}} - \Div V I  \bigr) \nabla u \cdot \nabla p  \dx + 
t \int_{\Omega} \Div( f V) p  \dx  \nonumber \\
& & +  c t^2 +c \Vert \hat u_t - u \Vert_{H^1}^2 + c t  \sup_{0 \leq \sigma \leq t} \Vert \nabla f \circ T_\sigma - \nabla f \Vert_{L^2} -t\int_{\Omega} g(u)p \mbox{div}Vdx. \label{t24}
\end{eqnarray}

In conclusion,
\begin{eqnarray}
T_2 & \leq &  t  \int_{\Omega} j_x(\cdot,u,\nabla u) \cdot V  \dx - t \int_{\Omega} j_z(\cdot,u,\nabla u) \cdot DV^{\mathsf{T}} \nabla u \nonumber \\
& & +  t \int_{\Omega} \bigl( D V + D V^{\mathsf{T}} - \Div V I  \bigr) \nabla u \cdot \nabla p  \dx + 
t \int_{\Omega} \Div( f V) p -t\int_{\Omega} g(u)p \mbox{div}V  \dx  \nonumber \\
& & +  c t^2 +c \Vert \hat u_t - u \Vert_{H^1}^2 + c t  \sup_{0 \leq \sigma \leq t} \Vert \nabla f \circ T_\sigma - \nabla f \Vert_{L^2}. \label{t2a}
\end{eqnarray}
Finally, $T_3$ can be estimated as in \cite[(3.19)]{DHH25-convergence},
so that 
\begin{equation} \label{t3}
T_3 \leq ct \bigl( t + \Vert \hat u_t - u \Vert_{H^1} \bigr) \leq c t^2 + c \Vert \hat u_t - u \Vert_{H^1}^2.
\end{equation}
If we insert the estimates \eqref{eq:t1}, \eqref{t2a} and \eqref{t3} into \eqref{difj} and recall \eqref{firstvar} we obtain
\begin{equation} \label{difj1}
\mathcal J(\Omega_t) - \mathcal J(\Omega)  \leq    t \mathcal J'(\Omega)[V]  + ct \bigl( t  +  \sup_{0 \leq \sigma \leq t} \Vert \nabla f \circ T_\sigma - \nabla f \Vert_{L^2} \bigr)
+ c \Vert \hat u_t - u \Vert^2_{H^1}.
\end{equation}
It remains to bound $\Vert \hat u_t - u \Vert_{H^1}$. To do so,  we use \eqref{state} and \eqref{stateuht} and derive
\begin{eqnarray*}
\lefteqn{\int_{\Omega} A_t  \nabla ( \hat u_t - u) \cdot \nabla \eta  \dx + \int_{\Omega}\big(g(\hat u_t)-g(u)\big)\eta \dx}\\ &=& \int_{\Omega} (I-A_t) \nabla u \cdot \nabla \eta  \dx + \int_{\Omega}g(\hat u_t)\eta (1-\mbox{ det }DT_t) \dx + \int_{\Omega} \bigl(
f \circ T_t  \, \mbox{det} D T_t -f \bigr) \eta  \dx
\end{eqnarray*}
for all $\eta \in H^1_0(\Omega)$.
It follows from  \eqref{idmat} that there exists $0< \delta_2 \leq \delta_1$ such that $A_t \xi \cdot \xi \geq \frac{1}{2} | \xi |^2$ for all $\xi \in \mathbb R^d$ and $0 \leq t \leq \delta_2$. Inserting $\eta=
\hat u_t-u$ into the above relation and using \eqref{idmat} as well as the
fact that $g' \geq 0$  we obtain
\begin{equation} \label{eq:hatu}
\frac{1}{2} \int_D |  \nabla (\hat u_t - u) |^2  \dx   \leq  c t \Vert \nabla u \Vert_{L^2} \Vert \nabla (\hat u_t - u) \Vert_{L^2}
+ c t(1+ \Vert  f \circ T_t  \, \mbox{det} D T_t -f \Vert_{L^2})  \Vert \hat u_t -u \Vert_{L^2}.
\end{equation}
One easily infers from \cite[(3.13)]{DHH25-convergence} that
\begin{displaymath}
\Vert  f \circ T_t  \, \mbox{det} D T_t -f \Vert_{L^2} \leq c
\Vert f \Vert_{H^1},
\end{displaymath}
so that \eqref{eq:hatu} along with Poincar\'e's inequality yields
$ \Vert \hat u_t - u \Vert_{H^1} \leq ct$.
Using this estimate in 
\eqref{difj1} and applying the assumption  $\mathcal J'(\Omega)[V] \leq - \epsilon$ we obtain
\begin{equation} \label{eq:jdecrease}
\mathcal J(\Omega_t) - \mathcal J(\Omega) 
 \leq  \gamma t \mathcal J'(\Omega)[V] -(1-\gamma) \epsilon t   + c t^2  + c t  \sup_{0 \leq \sigma \leq t} \Vert \nabla f \circ T_\sigma - \nabla f \Vert_{L^2}. 
\end{equation}
Arguing as after (3.21) in \cite{DHH25-convergence} we obtain
the existence of $0< \delta \leq \delta_2$ such that
\begin{displaymath}
\sup_{0 \leq \sigma \leq t} \Vert \nabla f \circ T_\sigma - \nabla f \Vert_{L^2} \leq 
\frac{1}{2c} (1-\gamma)   \epsilon \quad \mbox{ for } 0 \leq t \leq  \delta.
\end{displaymath}
Thus we deduce from \eqref{eq:jdecrease} for $0 \leq t \leq \delta$
\begin{displaymath}
\mathcal J(\Omega_t) - \mathcal J(\Omega)
 \leq   \gamma t \mathcal J'(\Omega)[V] -(1-\gamma) \epsilon t
 + c t \delta + \frac{1}{2} t(1-\gamma) \epsilon \leq 
 \gamma t \mathcal J'(\Omega)[V]
\end{displaymath}
provided one chooses in addition $\delta \leq \frac{1}{2c}(1-\gamma)
\epsilon$. 
\end{proof}

\noindent
We are now in position to prove our first convergence result.

\begin{thm} \label{conv1}
Let $(\Phi^k)_{k \in \mathbb N_0}$ and $(\Omega^k=\Phi^k(\Omref))_{k \in \mathbb N_0} \subset \mathcal S$ be the sequence generated by Algorithm \ref{alg:Steepest}. 
Then  $\Vert \mathcal J'(\Omega^k) \Vert \rightarrow 0$ as $k \rightarrow \infty$.
\end{thm}
    
\begin{proof} Let us assume that $\Vert \mathcal J'(\Omega^k) \Vert$ 
does not converge to zero. Then there exists $\epsilon>0$ and a subsequence $(\Omega^{k_j})_{j \in \mathbb N}$
such that $\Vert \mathcal J'(\Omega^{k_j}) \Vert \geq \epsilon$ for all  $j\in \mathbb N$. According to step 2 in Algorithm \ref{alg:Steepest}
let $V^{k_j} \in W^{1,\infty}_0(D,\mathbb R^d)$ such that $| D V^{k_j} | \leq 1$ a.e. in $D$ and 
\begin{equation} \label{leps}
\mathcal J'(\Omega^{k_j})[V^{k_j}] = -  \|\mathcal{J}'(\Omega^{k_j})\| \leq - \epsilon \quad \mbox{ for all } j \in \mathbb N.
\end{equation}
We infer from Lemma \ref{decrease} that there exists $\delta>0$ which is independent of $j \in \mathbb N$ such that
\begin{displaymath}
\mathcal J\bigl( (\mbox{id}+ t V^{k_j})(\Omega^{k_j}) \bigr)  - \mathcal J(\Omega^{k_j}) \leq \gamma t \mathcal J'(\Omega^{k_j})[V^{k_j}] 
 \qquad \mbox{ for all } 0 \leq t \leq \delta.
\end{displaymath}
Thus the Armijo step size  $t_{k_j} $ chosen in step 3 of
the algorithm satisfies $t_{k_j} \geq \frac{\delta}{2}$ for all $j \in \mathbb N$ so that we obtain with
$\Omega^{k_j+1} = (\mbox{id}+t_{k_j} V^{k_j})(\Omega^{k_j})$
and \eqref{leps} that
\begin{displaymath}
\mathcal J(\Omega^{k_j}) - \mathcal J(\Omega^{k_{j}+1}) \geq -\gamma t_{k_j} \mathcal J'(\Omega^{k_j})[V^{k_j}]  \geq \gamma t_{k_j} \epsilon
 \geq \gamma \frac{\delta}{2} \epsilon \qquad  \mbox{ for all } j \in \mathbb N.
\end{displaymath}
Combining this bound with the fact that 
$\mathcal J(\Omega^{k+1}) \leq \mathcal J(\Omega^k)$ 
we deduce for $J \in \mathbb N$
\begin{displaymath}
J \gamma \frac{\delta}{2} \epsilon \leq \sum_{j=1}^J (
\mathcal J(\Omega^{k_j}) - \mathcal J(\Omega^{k_{j}+1}))
\leq \sum_{k=0}^{k_J} (\mathcal J(\Omega^k) - \mathcal J(\Omega^{k+1})) = \mathcal J(\Omega^0) - \mathcal J(
\Omega^{k_J+1}).
\end{displaymath}
Thus we obtain a contradiction for large enough $J$ since \eqref{jest} yields
\begin{displaymath}
\mathcal J(\Omega^k)  \geq   - \int_{\Omega^k} | j(\cdot,u_k,\nabla u_k) |  \dx \geq - \int_{\Omega^k} \left( \varphi_1+ c_1 \bigl(  | u_k |^q + | \nabla u_k |^2 \bigr) \right)  \dx  
 \geq  - c \bigl( 1+ (c^*)^q  \bigr)
\end{displaymath}
where $u_k$ denotes the solution of \eqref{state} in $\Omega^k$ and \eqref{discapriori} was used. 
\end{proof}

\noindent
\subsection{Convergence to a stationary shape in two dimensions}

In what follows we restrict our analysis to the two--dimensional case and assume in addition that
the reference domain $\Omref$ has a Lipschitz boundary. 
In order to examine the convergence of the sequence $(\Omega^k)_{k \in \mathbb N_0}$ with
$\Omega^k=\Phi^k(\Omref)$ we require the following concepts, see
e.g. \cite{HenMic06}. 
Given two open sets  $\Omega_1, \Omega_2 \subset D$ their Hausdorff complementary distance 
is defined by
\begin{displaymath}
\rho_H^c(\Omega_1,\Omega_2):= \max_{x \in \bar D} | d_{\complement \Omega_1}(x) - d_{\complement \Omega_2}(x) |,
\end{displaymath}
where $d_{\complement \Omega}(x):= \inf\{|x-y| : y \in \bar D \setminus \Omega\}$ for all $x \in D$. 
We say that the sequence of open sets $(\Omega^k)_{k \in \mathbb N_0}$  converges to the open set $\Omega$ in the sense of the Hausdorff complementary metric if $\rho_H^c(\Omega^k,\Omega) \rightarrow 0,
k \rightarrow \infty$. Since our optimisation problem is constrained by the elliptic boundary value problem \eqref{state} we shall require   continuity of \eqref{state} with respect to $\Omega$ in an appropriate sense. In order to formulate the corresponding concept  we shall consider
$H^1_0(\Omega)$ as a closed subspace of $H^1_0(D)$ by associating with each element $u \in H^1_0(\Omega)$ its extension by zero $e_0(u) \in H^1_0(D)$. 

\begin{defn} Let $\Omega^k, \Omega$ be open subsets of $D$. \\
a) We say that  $(\Omega^k)_{k \in \mathbb N_0}$ $\gamma$--converges to $\Omega$, if $e_0(u_k) \rightarrow e_0(u)$
in $H^1_0(D)$ for every $f \in H^{-1}(D)$. Here, $u_k,u$ denote the solutions of
\[\int_{\Omega^k}\nabla u_k \cdot \nabla \eta \dx = \langle f,\eta\rangle \quad  \forall \eta \in H^1_0(\Omega^k) \quad \text{ and } \quad \int_{\Omega}\nabla u \cdot \nabla \eta \dx = \langle f,\eta\rangle \quad  \forall \eta \in H^1_0(\Omega), \,\]respectively. Furthermore,  $\langle \cdot, \cdot \rangle$ is  the duality pairing between $H^{-1}$ and $H^1_0$. \\[2mm]
b) We say that $(H^1_0(\Omega^k))_{k \in \mathbb N_0}$ is Mosco--convergent to $H^1_0(\Omega)$ if 
\begin{itemize}
\item for every $u \in H^1_0(\Omega)$ there exists a sequence $(u_k)_{k \in \mathbb N}$ with $u_k\in H^1_0(\Omega^k)$ such that $e_0(u_k) \rightarrow e_0(u)$ in $H^1_0(D)$;
\item for every sequence $(u_k)_{k \in \mathbb N}$ such that $u_k \in  H^1_0(\Omega^k)$ and $e_0(u_k)\rightharpoonup v \in H^1_0(D)$ we have that $v\in H^1_0(\Omega)$.
\end{itemize}
\end{defn}

\begin{thm} \label{gM}  (\cite[Proposition 3.5.5]{HenPie18}) Let $(\Omega^k)_{k \in \mathbb N_0}$ be
a sequence of open subsets of $D$. Then $(\Omega^k)_{k \in \mathbb N_0}$ $\gamma$--converges to $\Omega$
if and only if $(H^1_0(\Omega^k))_{k \in \mathbb N_0}$ is Mosco--convergent to $H^1_0(\Omega)$.
\end{thm}

\noindent
Next, for an open set $\Omega \subset D$
we denote by $\sharp \Omega^c$ the number of connected components of $\bar D \setminus \Omega$. 

\begin{thm} \label{sv} ({\v{S}}ver{\'a}k, \cite{vsverak1993optimal}). Let $d=2$ and $(\Omega^k)_{k \in \mathbb N_0}$ a sequence  such that $\rho_H^c(\Omega^k,\Omega) \rightarrow 0, k \rightarrow \infty$ for some open set $\Omega$. If  
$\sharp (\Omega^k)^c$ is uniformly bounded, then $(\Omega^k)_{k \in \mathbb N_0}$ $\gamma$--converges to $\Omega$.
\end{thm}

\vspace{2mm} 
The above result will be crucial in order to pass to the limit
in the formula for the shape derivative. Apart from the convergence
of the state and adjoint state we will also require the
convergence of the indicator functions $\chi_{\Omega^k}$ to $\chi_{\Omega}$. This will be a consequence of the following
result:

\begin{thm} \label{cd} (Chambolle \& Doveri, \cite{ChamDov97}). Let $d=2$ and $(\Omega^k)_{k \in \mathbb N_0}$ a sequence of open sets such that $\rho_H^c(\Omega^k,\Omega) \rightarrow 0, k \rightarrow \infty$ for some open set $\Omega$. Assume that the number of connected components of $\partial \Omega^k$ is uniformly bounded and that  $\sup_{k \in \mathbb N_0} \mathcal H^1(\partial \Omega^k)<\infty$. For $f \in L^2(D)$ denote by $u_\Omega \in H^1(\Omega)$ the weak solution of the Neumann problem 
\begin{displaymath}
    -\Delta u_\Omega + u_\Omega = f \mbox{ in } \Omega, \quad  \frac{\partial u_\Omega}{\partial \nu} =0 \mbox{ on } 
    \partial \Omega
\end{displaymath}
and by $\bar u_\Omega, \bar \nabla u_\Omega$ the trivial extensions of $u_\Omega, \nabla u_\Omega$
to $D$. Then $(\bar u_{\Omega^k},\bar \nabla u_{\Omega^k})  \rightarrow (\bar u_{\Omega}, \bar \nabla u_{\Omega})$ in $L^2(D)^3$.
\end{thm}

\vspace{2mm}
\noindent
Let us next formulate the main result of this section:

\begin{thm}\label{conv2}
Let $(\Omega^k)_{k \in \mathbb N_0}$ with $\Omega^k=\Phi^k(\Omref)$ be the sequence generated by Algorithm \ref{alg:Steepest} and suppose
that $(\Phi^k)_{k \in \mathbb N_0}$ is bounded in $W^{1,\infty}(D,\mathbb R^2)$. Then, there exists a subsequence $(\Omega^{k_j})_{j \in \mathbb N}$ which converges to some shape $\Omega \in \mathcal S$ (possibly $\Omega=\emptyset)$ with respect to the Hausdorff complementary metric,  
and $\mathcal J'(\Omega)[V] = 0$ for each $V \in W^{1,\infty}_0(D,\mathbb R^2)$, i.e. $\Omega$ is a stationary shape (where we formally set $\mathcal J'(\emptyset)=0$).
\end{thm}
\begin{proof} Our assumption implies that there exists $M \geq 0$ such that $\Vert \Phi^k \Vert_{W^{1,\infty}(D,\mathbb R^2)} \leq M$ and hence in view of the convexity of $D$
\begin{equation} \label{eq:lipbound}
| \Phi^k(x_1)-\Phi^k(x_2) | \leq M | x_1-x_2| \qquad \mbox{for all } x_1,x_2 \in \bar D \mbox{ and all } k \in \mathbb N_0.
\end{equation}
By the theorem of Arzel\`a-Ascoli there exist a subsequence, again denoted by $(\Phi^k)_{k\in\mathbb{N}_0}$, and $\Phi \in C^0(\bar D,\mathbb{R}^2)$ such that $\Phi^k \rightarrow \Phi$ uniformly.
Let us define $\Omega:= D \setminus \Phi(\bar D \setminus \Omref)$. Clearly $\Omega$ is open and
$\Omega \subset D$ so that $\Omega \in \mathcal S$. Let us assume for the remainder of the proof that $\Omega \neq \emptyset$. We claim that
\begin{equation} \label{eq:hconv}
\rho_H^c(\Omega^k,\Omega) \leq \Vert \Phi^k - \Phi \Vert_{C^0(\bar D;\mathbb R^d)}.
\end{equation}
To see this, let $x \in \bar D$ and choose $y \in \bar D \setminus \Omega$ such that
$d_{\complement \Omega}(x) = | x-y|$. Since $\Phi^k=\id$ on $\partial D$ for all $k \in \mathbb N_0$, we have $\Phi=\id$ on $\partial D$ 
so that  $\partial D \subset \Phi(\bar D \setminus \Omref)$ and therefore
$\bar D \setminus \Omega = \Phi(\bar D \setminus \Omref)$. Thus 
there exists $z \in \bar D \setminus \Omref$ such that $\Phi(z)=y$. As $\Phi^k$ is bijective and $\Omega^k=
\Phi^k(\Omref)$ we have that  $y_k:= \Phi^k(z) \in \bar D\setminus \Omega^k$ and hence
\begin{displaymath}
d_{\complement \Omega^k}(x)  - d_{\complement \Omega}(x) \leq | x - y_k | - | x- y| \leq | y_k - y |
= | \Phi^k(z) - \Phi(z) | \leq \Vert \Phi^k - \Phi \Vert_{C^0(\bar D;\mathbb R^d)}.
\end{displaymath}
By exchanging the roles of $\Omega^k$ and $\Omega$ we deduce \eqref{eq:hconv}. 
As a result, the sequence $(\Omega^k)_{k \in \mathbb N_0}$ converges to $\Omega$ with respect to
the Hausdorff complementary metric. By Lemma \ref{lem:propalg},  $\Phi^k: \bar D \rightarrow \bar D$ is a homeomorphism and $\bar D \setminus \Omega^k= \bar D \setminus \Phi^k(\Omref)= \Phi^k(\bar D \setminus \Omref)$ and hence we infer that $\sharp (\Omega^k)^c= \sharp \Omref^c$ 
is uniformly bounded. Thus, Theorem \ref{sv} implies that
$(\Omega^k)_{k \in \mathbb N_0}$ $\gamma$-- converges to $\Omega$.
In view of Lemma \ref{lem:apriori} we have that
\begin{equation} \label{bounds}
\Vert e_0(u_k) \Vert_{H^1} + \Vert e_0(u_k) \Vert_{L^\infty} \leq c^*, \quad k \in \mathbb N,
\end{equation}
so that there exist a subsequence, again not relabeled, and $v \in H^1_0(D) \cap L^\infty(D)$ such that
\begin{equation}\label{ucon}
e_0(u_k)  \rightharpoonup v \text{ in } H^1_0(D), \quad e_0(u_k)  \rightarrow v \text{ in } L^2(D), \quad
e_0(u_k) \rightarrow v \text{ a.e. in } D.
\end{equation}
In view of Theorem \ref{gM} we have that $(H^1_0(\Omega^k))_{k \in \mathbb N_0}$ Mosco--converges to $H^1_0(\Omega)$
which  implies that  $v \in H^1_0(\Omega)$. Our next aim is to prove that  $e_0(u_k) \rightarrow v$  in $H^1_0(D)$, and that $v$ solves the state equation \eqref{state} in $\Omega$. To begin, denote by $\tilde u_k \in H^1_0(\Omega^k)$ the unique solution of the linear equation 
\[
\int_{\Omega^k} \nabla \tilde u_k \cdot \nabla \eta \dx = \int_{\Omega^k} (f-g(v))\eta \dx \quad \forall \, \eta \in H^1_0(\Omega^k).
\]
Then, the $\gamma-$convergence implies that $e_0(\tilde u_k) \rightarrow e_0(u)$, 
in $H^1_0(D)$, where $u \in H^1_0(\Omega)$ solves
\begin{equation}\label{uprob}
\int_{\Omega} \nabla u \cdot \nabla \eta \dx = \int_{\Omega} (f-g(v))\eta \dx \quad \forall \, \eta \in H^1_0(\Omega).
\end{equation}
The difference $\tilde u_k-u_k$ satisfies
\[
\int_{\Omega^k} \nabla (\tilde u_k-u_k)\cdot \nabla \eta \dx = \int_{\Omega^k} (g(u_k)-g(v))\eta \dx \quad \forall \, \eta \in H^1_0(\Omega^k).
\]
Plugging  $\eta := \tilde u_k-u_k$ into the above relation 
and using the uniform boundedness of the sequence 
$(u_k)_{k \in \mathbb N_0}$ we obtain
\[
\Vert \tilde u_k-u_k\Vert_{H^1}^2 \le c \Vert u_k-v\Vert_{L^2}\Vert\tilde u_k-u_k\Vert_{L^2},
\]
which implies $\Vert \tilde u_k-u_k\Vert_{H^1} \rightarrow 0$. Thus, after possibly passing
to a further subsequence,
\begin{equation} \label{eq:limit1}
 e_0(u_k) \rightarrow e_0(u) \qquad \mbox{ in }  H^1_0(D) \mbox{ and a.e. in } D,
\end{equation}
so that by \eqref{ucon} $e_0(u)=v$ and hence it follows from \eqref{uprob} that $u$ solves the state equation \eqref{state}
in $\Omega$. 
Next we claim that 
\begin{eqnarray}
j_u(\cdot,e_0(u_k),\nabla e_0(u_k)) &\rightarrow & j_u(\cdot,e_0(u), \nabla e_0(u)) \quad \mbox{ in } L^{\frac{q}{q-1}}(D), \label{convph1} \\
j_z(\cdot,e_0(u_k),\nabla e_0(u_k)) & \rightarrow & j_z(\cdot,e_0(u), \nabla e_0(u)) \quad \mbox{ in } L^2(D,\mathbb R^d).\label{convph2}
\end{eqnarray}
In order to show \eqref{convph2} we set $f_k:=| j_z(\cdot,e_0(u_k),\nabla e_0(u_k)) - j_z(\cdot,e_0(u), \nabla e_0(u)) |^2$. Clearly, $f_k \rightarrow 0$ a.e.~in $D$, while
\eqref{j2est} implies that 
\begin{displaymath}
f_k \leq c \bigl( \varphi_3^2 + | e_0(u_k) |^q + | \nabla e_0(u_k) |^2 + | e_0(u) |^q + | \nabla e_0(u) |^2 \bigr)=: r_k.
\end{displaymath}
We have that $r_k \rightarrow r:=c \bigl( \varphi_3^2 + 2 | e_0(u) |^q + 2| \nabla e_0(u) |^2 \bigr)$ a.e.~in $D$ as well as $\int_D r_k \dx \rightarrow \int_D r \dx$ as $k
\rightarrow \infty$, so that the generalised Lebesgue dominated convergence theorem yields \eqref{convph2}.
The relation \eqref{convph1} is proved in the same way. \\
Let us next consider the sequence of adjoint solutions $(p_k)_{k \in \mathbb N_0}$. Similar as
above we may assume that there exists a further subsequence 
and $w \in H^1_0(D)$ such that
\begin{equation}\label{wcon}
e_0(p_k)  \rightharpoonup w \text{ in } H^1_0(D), \quad e_0(p_k)  \rightarrow w \text{ in } L^2(D), \quad 
e_0(p_k) \rightarrow w \text{ a.e. in } D.
\end{equation}
As above our  aim is to prove that $e_0(p_k) \rightarrow w$ strongly in $H^1_0(D)$ and that $w$ solves the adjoint equation \eqref{adj} in $\Omega$. Let
$h \in H^{-1}(D)$ be defined by
\begin{displaymath}
\langle h,\eta \rangle:= \int_{D} \bigl( j_u(\cdot,e_0(u),\nabla e_0(u)) \eta  + j_z(\cdot,e_0(u),\nabla e_0(u)) \cdot \nabla \eta \bigr)  \dx
\end{displaymath}
and denote by $\tilde p_k \in H^1_0(\Omega^k)$ the solution of
\begin{displaymath}
\int_{\Omega^k} \nabla \tilde p_k \cdot \nabla \eta  \dx  = \langle h,\eta \rangle
-  \int_{\Omega^k} g'(e_0(u)) w \eta \dx \qquad \forall \eta \in H^1_0(\Omega^k).
\end{displaymath}
Using that $(\Omega^k)_{k \in \mathbb N_0}$ $\gamma$--converges to $\Omega$ we deduce that $e_0(\tilde p_k) \rightarrow e_0(p)$ in $H^1_0(D)$ as $k \rightarrow \infty$,
where $p \in H^1_0(\Omega)$ is the unique solution of 
\begin{displaymath}
\int_\Omega \nabla p \cdot \nabla \eta  \dx = \langle h, \eta \rangle 
-  \int_{\Omega} g'(e_0(u)) w \eta \dx\qquad\forall \eta \in H^1_0(\Omega).
\end{displaymath}
On the other hand we have that
\begin{eqnarray*}
\lefteqn{ 
\int_{\Omega^k} \nabla (p_k - \tilde p_k) \cdot \nabla \eta  \dx =
\int_{\Omega^k} \bigl( g'(e_0(u)) w - g'(u_k) p_k \bigr) \eta \dx }  \\
& & +   \int_{\Omega^k} \bigl( j_u(\cdot,e_0(u_k),\nabla e_0(u_k))-
j_u(\cdot,e_0(u),\nabla e_0(u)) \bigr) \, \eta  \dx  \\
& & + \int_{\Omega^k} \bigl( j_z(\cdot,e_0(u_k),\nabla e_0(u_k)) - j_z(\cdot,e_0(u), \nabla e_0(u)) \bigr) \cdot \nabla \eta  \dx \qquad\forall \eta \in H^1_0(\Omega^k).
\end{eqnarray*} 
Testing the above relation with $\eta=p_k - \tilde p_k$ and
recalling \eqref{eq:limit1}, \eqref{convph1}, \eqref{convph2} and \eqref{wcon} we infer that 
$e_0(p_k) - e_0(\tilde p_k) \rightarrow 0$ in $H^1_0(D)$ so that in conclusion 
\begin{equation} \label{eq:limit2}
e_0(p_k) \rightarrow e_0(p) \qquad \mbox{ in }  H^1_0(D), 
\end{equation}
and thus $w=e_0(p)$, i.e. $p$ solves the adjoint equation \eqref{adj}.
Next we note that $\partial \Omega^k = \Phi^k(\partial \Omref)$, so that the number of connected components of $\partial \Omega^k$
coincides with the corresponding number of $\partial \Omref$ and
is therefore uniformly bounded. Furthermore, \eqref{eq:lipbound} and the fact that $\Omref$ has a Lipschitz boundary imply that
\begin{displaymath}
\sup_{k \in \mathbb N_0} \mathcal H^1(\partial \Omega^k) = \sup_{k \in \mathbb N_0}  \mathcal H^1(\Phi^k(\partial \Omref)) \leq M  \mathcal H^1(\partial \Omref) < \infty.
\end{displaymath}
Applying Theorem \ref{cd} with $f \equiv 1$ and noting that the solution of the corresponding Neumann problem on $\Omega$ is given by the indicator function $\chi_\Omega$ we infer that 
\begin{equation} \label{indicator}
\chi_{\Omega^k} \rightarrow \chi_{\Omega} \mbox{ in } L^2(D) \mbox{ and a.e. in  } D
\end{equation}
after possibly passing to a further subsequence. 
We have for $V \in W^{1,\infty}_0(D,\mathbb R^2)$
\begin{eqnarray}
\mathcal J'(\Omega^k)[V]  &= &  \int_{\Omega^k} 
\bigl( DV + D V^{\mathsf{T}} - \Div V I \bigr)  \nabla u_k \cdot \nabla p_k   \dx + 
\int_{\Omega^k} \Div( f V) p_k  - g(u_k)p_k \Div V \dx \nonumber \\
&  & +  \int_{\Omega^k} \Bigl( j(\cdot,u_k,\nabla u_k) \Div V +  j_x(\cdot,u_k,\nabla u_k) \cdot V - j_z(\cdot,u_k,\nabla u_k) \cdot DV^{\mathsf{T}} \nabla u_k \Bigr)  \dx  \nonumber  \\
& =: & A_k+B_k + C_k. \label{akbk}
\end{eqnarray} 
Abbreviating $A:=\int_{\Omega}  
\bigl( DV + D V^{\mathsf{T}} - \Div V I \bigr)  \nabla u \cdot \nabla p   \dx $ we have
\begin{eqnarray*}
\lefteqn{
| A_k - A |  = | \int_D \bigl( DV + D V^{\mathsf{T}} - \Div V I \bigr) \bigl( 
\chi_{\Omega^k} \nabla e_0(u_k) \cdot \nabla e_0(p_k) - \chi_{\Omega} \nabla e_0(u) \cdot \nabla e_0(p) 
\bigr) | }\\
& \leq & C \int_D \chi_{\Omega^k} | \nabla e_0(u_k) \cdot \nabla e_0(p_k) - \nabla e_0(u) \cdot \nabla e_0(p) |  \dx + C \int_D | \chi_{\Omega^k} - \chi_{\Omega} | \,  | \nabla e_0(u) \cdot \nabla e_0(p) |  \dx \\
& \leq & C \bigl( \Vert e_0(u_k) - e_0(u) \Vert_{H^1} \Vert e_0(p_k) \Vert_{H^1} +
\Vert e_0(u) \Vert_{H^1} \Vert e_0(p_k) - e_0(p) \Vert)_{H^1} \bigr)  +  C \int_D \zeta_k   \dx \\
& \leq & C \bigl( \Vert e_0(u_k) - e_0(u) \Vert_{H^1} +
 \Vert e_0(p_k) - e_0(p) \Vert)_{H^1} \bigr)  +  C \int_D \zeta_k   \dx,
\end{eqnarray*}
where  we have used Lemma \ref{lem:apriori} and have set $ \zeta_k:=| \chi_{\Omega^k} - \chi_{\Omega} | | \nabla e_0(u) \cdot \nabla e_0(p) |$. Clearly, $\zeta_k \leq 2 | \nabla e_0(u)| | \nabla e_0(p)| \in L^1(D)$ and $\zeta_k \rightarrow 0$ a.e. in $D$ in view of \eqref{indicator}. Thus Lebesgue's dominated convergence theorem together with  \eqref{eq:limit1} and \eqref{eq:limit2} implies that $\lim_{k \rightarrow \infty} A_k=A$, while a  similar argument
shows that $\lim_{k \rightarrow \infty} B_k=B:= \int_\Omega \Div (fV) p - g(u)p \Div V \dx$. Finally, in order to
examine the terms $C_k$ we note that similarly to \eqref{convph2} one can show that
\begin{displaymath}
j(\cdot,e_0(u_k),\nabla e_0(u_k))  \rightarrow j(\cdot,e_0(u), \nabla e_0(u)), \; 
j_x(\cdot,e_0(u_k),\nabla e_0(u_k))  \rightarrow  j_x(\cdot,e_0(u), \nabla e_0(u)) \; \mbox{ in } L^1(D),
\end{displaymath}
from which we infer together with \eqref{convph2} that 
\begin{displaymath}
\lim_{k \rightarrow \infty} C_k = C:= \int_{\Omega} \Bigl( j(\cdot,u,\nabla u) \Div V +  j_x(\cdot,u,\nabla u) \cdot V - j_z(\cdot,u,\nabla u) \cdot DV^{\mathsf{T}} \nabla u \Bigr)  \dx.
\end{displaymath}
Recalling \eqref{akbk} we deduce that $\mathcal J'(\Omega)[V]= \lim_{k \rightarrow \infty} \mathcal J'(\Omega^k)[V]  =0$ with the help of  Theorem \ref{conv1}.
\end{proof}

\section{Numerical experiment}
Finally, we present a numerical example that shows, on the one hand, that the conditions of Theorem \ref{conv2} can be verified numerically and are also fulfilled in our example, and on the other hand, that the iteration sequence of the steepest descent method in Algorithm \ref{alg:Steepest} can potentially converge to the empty set.
Here, we consider $g = \frac{1}{2}\exp$ and set the data $f(x) = 1 + g( u^*(x) )$, where $u^*(x) = \frac{4}{{3\pi}} - \frac{\|x\|^2}{4}$.
We also set $j(x,u,z) = \frac{1}{2}(u-u_d(x))^2$ with $u_d(x) = \frac{4}{\pi}-\|x\|^2$.
This setting is similar to that of \cite[Experiment 1]{DHH25-convergence} with a linear state equation, where the radial symmetric local minimizer $B_{\frac{4}{\sqrt{3\pi}}}(0)$ and $\emptyset$ of the shape functional are known. However, for the semilinear case considered here we are not aware of non-trivial exact (radially symmetric) solutions. But we expect that the solution behaviour will be similar to that in the linear case. 

The numerical experiments were conducted using DUNE \cite{duneReference}, particularly the python bindings \cite{DunePython1,DunePython2}.
The initial grids are constructed with pygmsh \cite{pygmsh}.

In our numerical experiment Algorithm \ref{alg:Steepest} is initialized with two different domains. As stopping criterion we use $\|J'(\Omega_k)\| \leq 10^{-3}$. If we initialize the steepest descent method with $\Omega^0 = (-0.75,0.75)^2$, the algorithm computes iterates $\Omega^k$ which seem to converge to the global minimizer $\emptyset$, compare Fig. \ref{fig:Meshes}, top row, whereas initialization with $\Omega^0 =(-1,1)^2$ generates a sequence $(\Omega^k)_{k\in\mathbb{N}_0}$ which seems to converge to a ball centered at zero close to the local minimizer $B_{\frac{4}{\sqrt{3\pi}}}(0)$ of the linear case, compare Fig. \ref{fig:Meshes}, bottom row.

To demonstrate that the norms $\|D\Phi_k\|_{L^\infty}$ stay bounded with increasing iteration counter $k$ we in the case $\Omega^0 = (-1,1)^2$ use a cascadic approach, where the triangulations of the respective domains are congruently refined after every 15 iterations of the algorithm. The result is displayed in Fig. \ref{fig:norms}, left for 4 refinement levels. It appears that the norm begins to stagnate as $k$ increases which clearly indicates that $\|D\Phi_k\|_{L^\infty}$ stays bounded as the refinement level increases. A similar behaviour is observed in the case $\Omega^0 = (-0.75,0.75)^2$ shown in Fig. \ref{fig:norms}, right, where the norms $\|D\Phi_k\|_{L^\infty}$ are depicted for refinement levels 1-4.

\begin{figure}
    \centering

    \includegraphics[width = .24\textwidth]{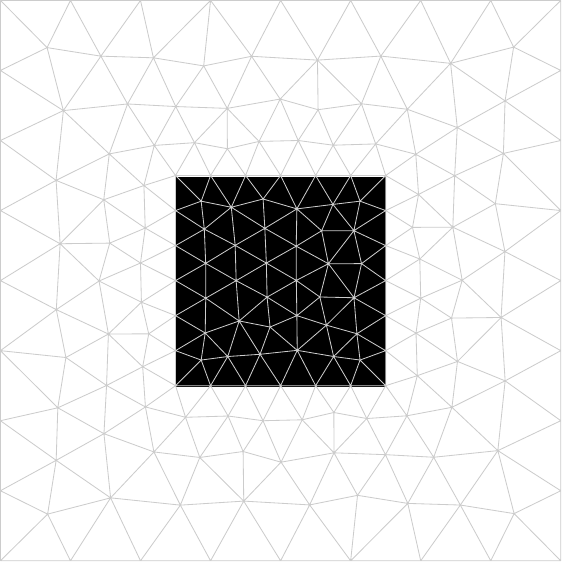}
    \includegraphics[width = .24\textwidth]{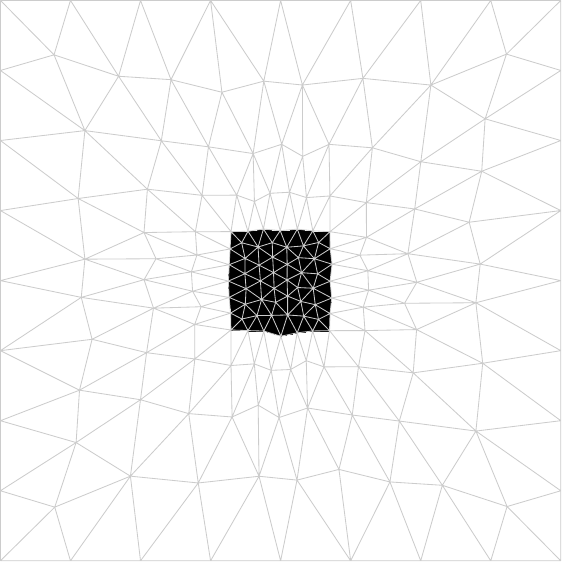}
    \includegraphics[width = .24\textwidth]{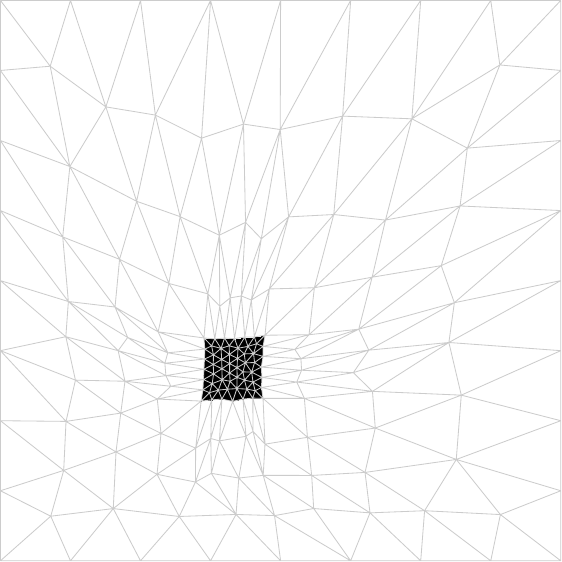}
    \includegraphics[width = .24\textwidth]{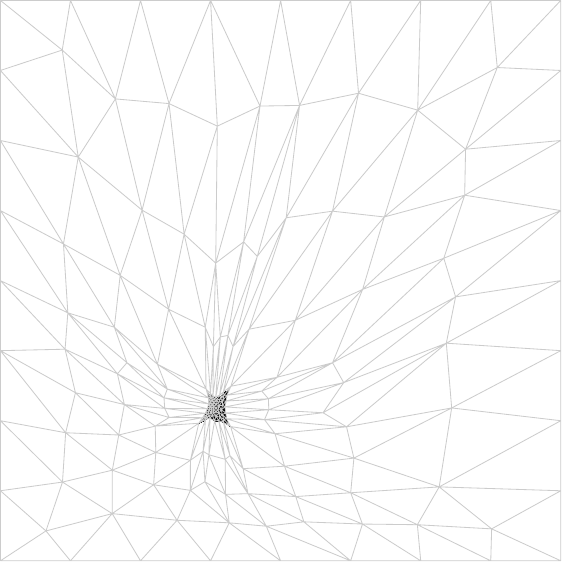}

    \includegraphics[width = .19\textwidth]{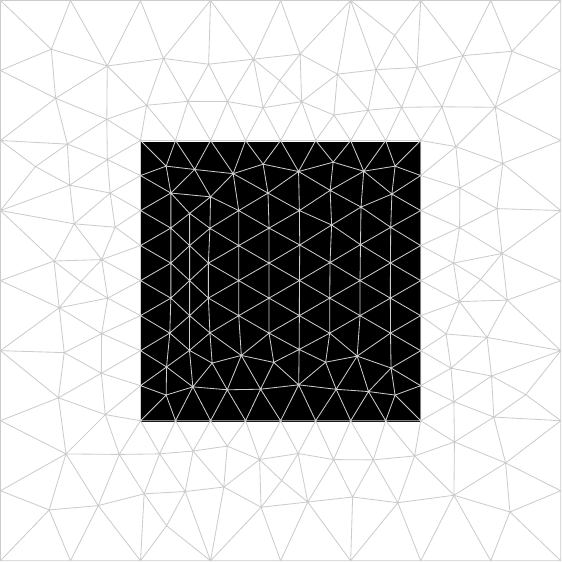}
    \includegraphics[width = .19\textwidth]{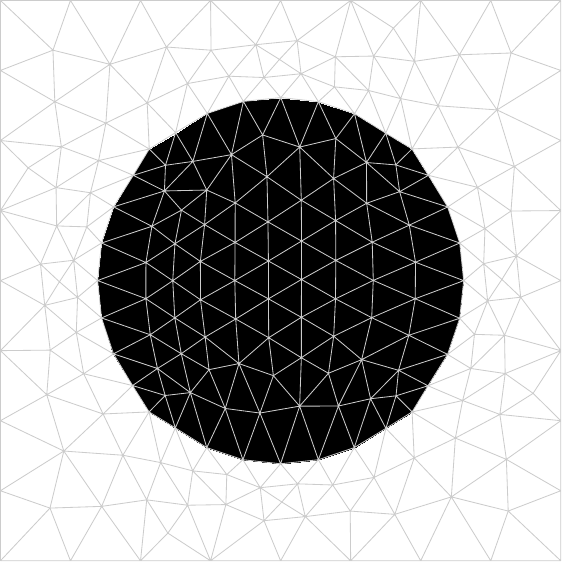}
    \includegraphics[width = .19\textwidth]{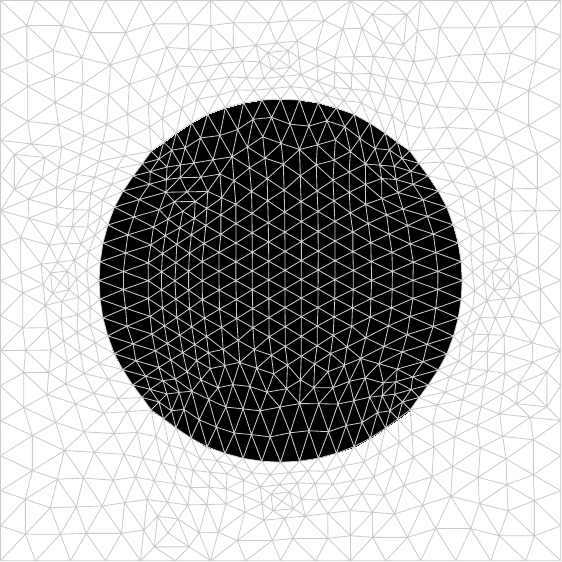}
    \includegraphics[width = .19\textwidth]{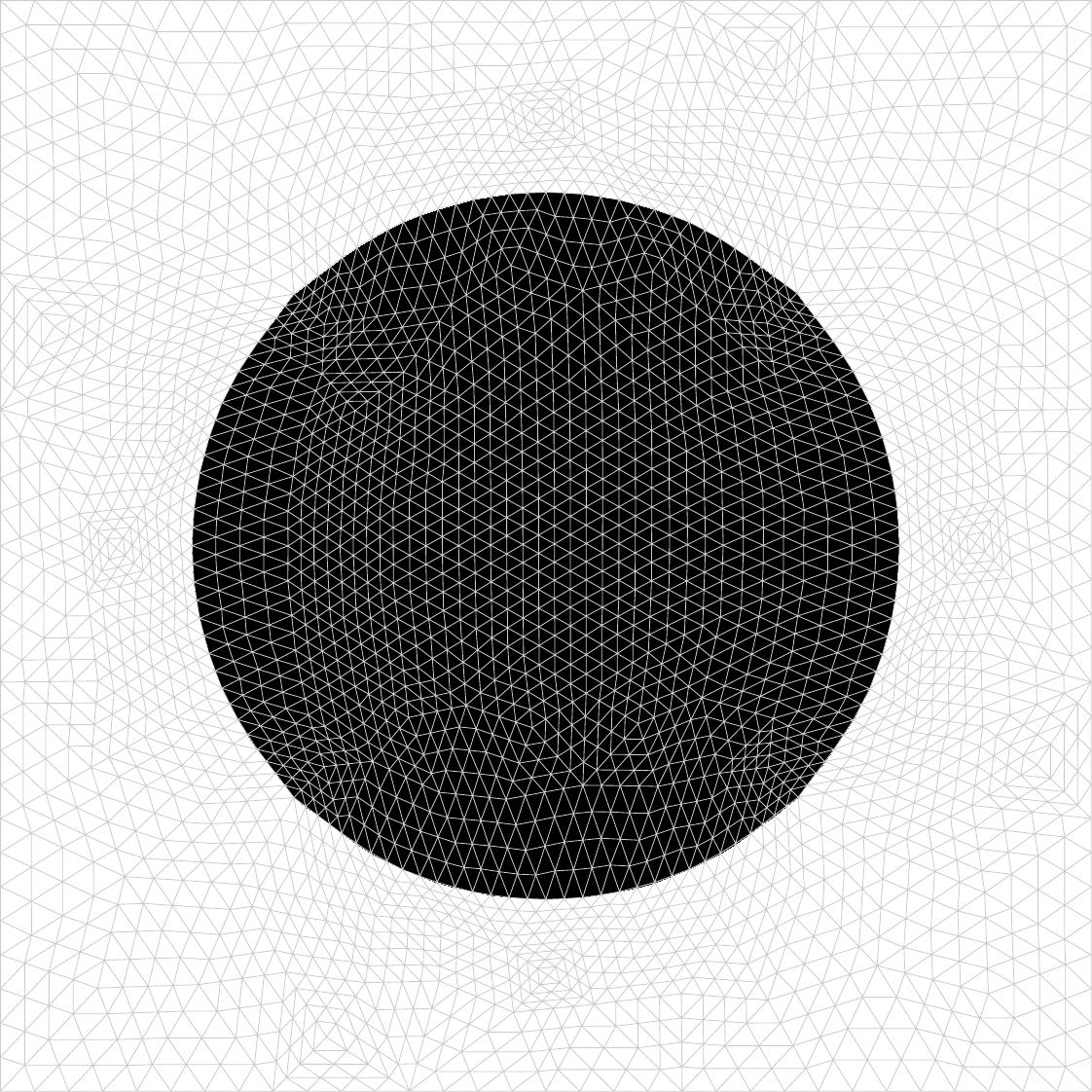}
    \includegraphics[width = .19\textwidth]{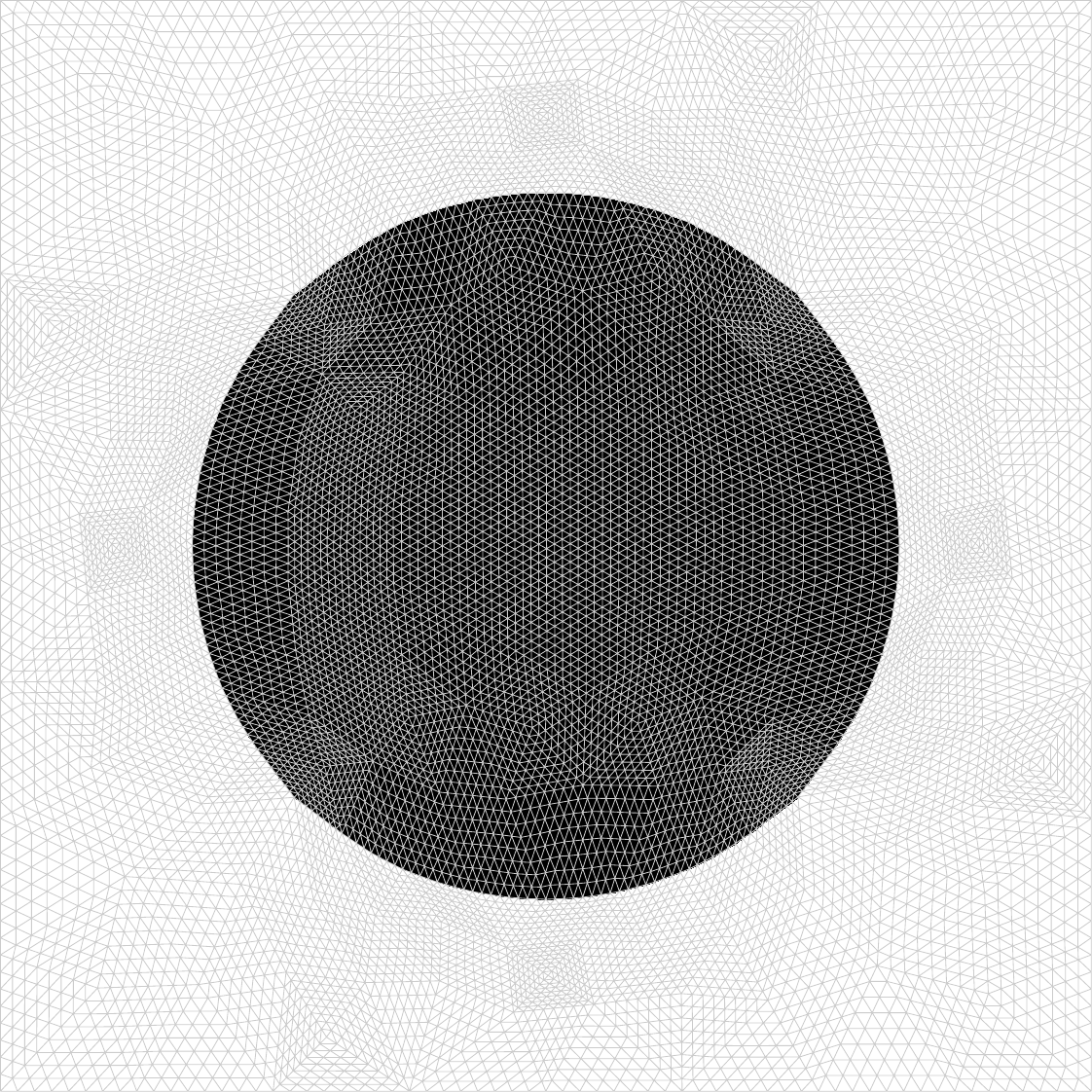}

\caption{Selected iterates of Algorithm \ref{alg:Steepest} when starting with $\Omega^0 = (-0.75,0.75)^2$ (top row), and with $\Omega^0 =(-1,1)^2$ (bottom row).  Black areas indicate $\Omega$. In the top row the domain together with the grid on refinement level 1 is shown. In the bottom row we display the domains together with the grids on refinement levels  1-4 (from left to right).}\label{fig:Meshes}
\end{figure}

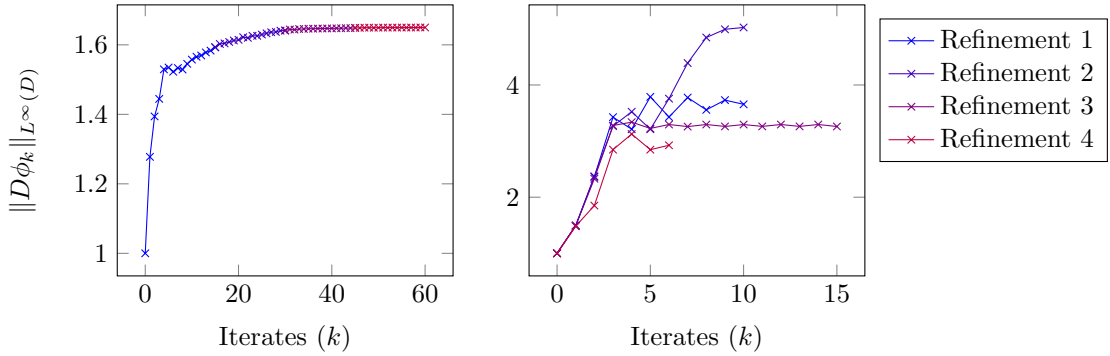
\begin{figure}
    \centering
    
    \begin{tikzpicture}
        \begin{groupplot}[group style={group size= 2 by 1},width=.4\linewidth, legend style={at={(2.95,0.95)}}]
            \nextgroupplot[xlabel = Iterates ($k$), ylabel = $\|D\phi_k\|_{L^\infty(D)}$]

            \addplot[mark=x, color = red!00!blue] table [x expr = \thisrow{iter} + 00, y=dPhi, col sep=comma, ] {expBall_0_semiLinearBall.csv};
            \addlegendentry{Refinement 1}
            \addplot[mark=x, color = red!25!blue] table [x expr = \thisrow{iter} + 15, y=dPhi, col sep=comma, ] {expBall_1_semiLinearBall.csv};
            \addlegendentry{Refinement 2}
            \addplot[mark=x, color = red!50!blue] table [x expr = \thisrow{iter} + 30, y=dPhi, col sep=comma, ] {expBall_2_semiLinearBall.csv};
            \addlegendentry{Refinement 3}
            \addplot[mark=x, color = red!75!blue] table [x expr = \thisrow{iter} + 45, y=dPhi, col sep=comma, ] {expBall_3_semiLinearBall.csv};
            \addlegendentry{Refinement 4}

            \nextgroupplot[xlabel = Iterates ($k$)]
            \addplot[mark=x, color = red!00!blue] table [x=iter, y=dPhi, col sep=comma, ] {expDegen_semiLinearDegen3.csv};
            \addplot[mark=x, color = red!25!blue] table [x=iter, y=dPhi, col sep=comma, ] {expDegen_semiLinearDegen4.csv};
            \addplot[mark=x, color = red!50!blue] table [x=iter, y=dPhi, col sep=comma, ] {expDegen_semiLinearDegen5.csv};
            \addplot[mark=x, color = red!75!blue] table [x=iter, y=dPhi, col sep=comma, ] {expDegen_semiLinearDegen6.csv};

        \end{groupplot}
    \end{tikzpicture}

    \caption{The evolution of $\|D\Phi_k\|_{L^\infty(D)}$ with the iteration counter $k$ for  $\Omega_0 = (-1,1)^2$ (non-degenerate case, left) and $\Omega_0 = (-0.75,0.75)^2$ (degenerate case, right).
    In the non-degenerate case, we see that $\|D\Phi_k\|_{L^\infty(D)}$ levels out.
    In the degenerate case, it is also seen that $\|D\Phi_k\|_{L^\infty(D)}$ levels out or the experiment stops due to the stopping criterion, or both.}
    \label{fig:norms}
\end{figure}

\section*{Acknowledgements}
The second and third author acknowledge funding of the project {\it 
Ein nichtglatter Phasenfeld Zugang für Formoptimierung mit instationären Fluiden} by the German Research foundation under project 423457678 within the Priority Programme 1962. The third author also acknowledges funding of the project {\it 
Fluiddynamische Formoptimierung mit Phasenfeldern und Lipschitz-Methoden} by the German Research foundation under project 543959359.

\vspace{\baselineskip}

\bibliographystyle{alpha}
\bibliography{bib}
\end{document}